\documentclass[a4paper,fleqn]{cas-sc}

\usepackage[numbers,sort&compress]{natbib}
\usepackage{amsmath,amssymb,amsfonts}
\usepackage{graphicx}
\usepackage[T1]{fontenc}
\usepackage{xcolor}
\usepackage{listings}
\usepackage{url}
\usepackage{eso-pic}
\AddToShipoutPictureFG*{\AtTextUpperLeft{\raisebox{2.2\baselineskip}[0pt][0pt]{%
\scriptsize\setlength{\fboxsep}{2pt}\fbox{\parbox{\dimexpr\textwidth-2\fboxsep-2\fboxrule\relax}{%
\textcolor{blue}{This is a preprint. The final version of this article is published as
R.~Monjo, \emph{Spacetime closedness theorem for homogeneous Lorentzian foliations},
Journal of Geometry and Physics \textbf{230} (2026) 105983,
\url{https://doi.org/10.1016/j.geomphys.2026.105983}.
Please cite the published version.}}}}}}

\definecolor{codegreen}{rgb}{0,0.45,0}
\definecolor{codegray}{rgb}{0.4,0.4,0.4}
\definecolor{codepurple}{rgb}{0.45,0,0.65}
\definecolor{backcolour}{rgb}{0.96,0.96,0.94}

\lstdefinestyle{leanstyle}{
backgroundcolor=\color{backcolour},
commentstyle=\color{codegreen},
keywordstyle=\color{magenta},
numberstyle=\tiny\color{codegray},
stringstyle=\color{codepurple},
basicstyle=\ttfamily\footnotesize,
breaklines=true,
keepspaces=true,
numbers=left,
numbersep=5pt,
showstringspaces=false,
frame=single,
rulecolor=\color{black},
framerule=0.4pt
}

\newtheorem{theorem}{Theorem}[section]
\newtheorem{lemma}[theorem]{Lemma}

\newtheorem{proposition}[theorem]{Proposition}
\newdefinition{definition}[theorem]{Definition}
\newdefinition{remark}[theorem]{Remark}
\newproof{proof}{Proof}

\begin{document}
\let\WriteBookmarks\relax
\def\floatpagepagefraction{1}
\def\textpagefraction{.001}
\shorttitle{Spacetime closedness theorem}
\shortauthors{R. Monjo}

\title [mode = title]{Spacetime closedness theorem for homogeneous Lorentzian foliations}

\author[1,2]{Robert Monjo}[orcid=0000-0003-3100-2394]
\cormark[1]
\ead{robert.monjo@cunef.edu}

\affiliation[1]{organization={Department of Mathematics, CUNEF Universidad},
addressline={Calle Almansa 101},
city={Madrid},
postcode={28040},
state={Madrid},
country={Spain}}

\affiliation[2]{organization={Centro de Estudios Cosmol\'ogicos ``Prof.\ Jaime Roessler Bonzi'', Facultad de Ciencias, Universidad de Chile},
addressline={Las Palmeras 3425},
city={Santiago},
postcode={7750000},
state={Santiago},
country={Chile}}

\cortext[cor1]{Corresponding author}

\begin{abstract}
Motivated by the geometric interpretation of spatially homogeneous cosmological models, we formulate a spacetime closedness theorem directly at the Lorentzian level. A classical space-form classification theorem organizes homogeneous and isotropic spatial geometries, but by itself it does not yield a Lorentzian statement about the ambient spacetime. The main technical step is to derive finite slice-volume from genuinely Lorentzian control hypotheses on the foliation. This is achieved in a strong version, for maximally regular globally hyperbolic $(n+1)$--spacetimes with a finite-time Big Bang and homogeneous complete spacelike slices, and in a weaker version in which maximal regularity is replaced by time-integrability of the accumulated expansion rate. In both cases, the argument separates an analytic step, deriving finite slice-volume from the Big Bang and temporal control, from a geometric step, upgrading finite volume to compactness by homogeneity and completeness. The theorem is stated in arbitrary spacetime dimension and is accompanied by a Lean~4 formalization of the strong and weak abstract statements.
\end{abstract}


\maketitle

\section{Introduction}

\subsection{Motivation}

Homogeneous and isotropic cosmological models are classically organized through the Friedmann--Lema\^itre--Robertson--Walker (FLRW) metrics; see for example \cite{HawkingEllis1973,ONeill1983,Wald1984}. In that framework, one often starts from a classical space-form classification result: complete simply connected homogeneous and isotropic Riemannian $3$-manifolds are space forms and, after normalization, have sectional curvature $K=-1,0,1$ \cite{ONeill1983,Wolf2011}. In modern presentations, this statement may be viewed as combining a Schur-type constant-curvature argument with the Killing--Hopf classification of complete simply connected space forms \cite{Wolf2011,Petersen2016}. Geometric properties of spacelike hypersurfaces in Robertson--Walker and generalized Robertson--Walker spacetimes, including completeness and higher-order mean curvature conditions, have been studied from different viewpoints; see, for example, \cite{AlbujerCamargoLima2011,AliasImperaRigoli2012}. Closely related aspects of generalized Robertson--Walker geometry, curvature, and causal structure were also studied in \cite{Sanchez1999,AledoGalvezRomero2004,Minguzzi2009}. Recent work in Lorentzian geometry has also emphasized time functions, Lorentzian length spaces, and canonical structures associated with globally hyperbolic spacetimes \cite{BurtscherGarciaHeveling2025,KunzingerSteinbauer2022,FinsterMuch2023}.

This theorem is mathematically correct in its proper setting, but it has a limited scope: it classifies spatial Riemannian manifolds, not Lorentzian spacetimes. Here a Lorentzian spacetime is a four-dimensional geometric model in which temporal and spatial directions are encoded simultaneously in a single metric structure. Such a metric is not merely a family of spatial metrics: it also contains the temporal and causal structure of the universe \cite{BernalSanchez2005,Monjo2024a,Monjo2024b}. Therefore, a classification theorem for spatial slices alone cannot, without additional hypotheses, yield a classification of the ambient Lorentzian geometry.

This distinction is substantive. From the viewpoint of global Lorentzian geometry, closedness or compactness conclusions at spacetime level must be derived from hypotheses that already involve a time function, a spacelike foliation, and quantitative control of the foliation itself. The main technical contribution is to identify Lorentzian hypotheses under which a finite-lifetime assumption can be converted into a compactness conclusion for homogeneous complete slices. The proof splits into two steps: an analytic step, in which finite slice-volume is derived from the Big Bang assumption and bounded-geometry control of the foliation; and a geometric step, in which finite volume is upgraded to compactness using homogeneity and completeness. For cosmological motivation the discussion is phrased in four spacetime dimensions. The compactness theorem proved later, however, depends only on the existence of a Lorentzian foliation by homogeneous complete spatial slices and therefore has the same formal content in arbitrary spacetime dimension $(n+1)$. 

\subsection{The spatial classification theorem and its scope}

\vspace{5mm}

The classical space-form classification theorem states that a complete simply connected homogeneous and isotropic Riemannian manifold is a space form of constant sectional curvature; see, for example, \cite[Chs.~2]{Wolf2011} together with the Schur-type argument discussed in \cite[Ch.~5]{Petersen2016}. Since this conclusion rests on the geometric meaning of homogeneity and isotropy, we recall these two notions explicitly:

\begin{definition}[Homogeneity and isotropy]
A Riemannian manifold $(\Sigma,h)$ is called \emph{homogeneous} if for every pair of points $p,q\in\Sigma$, there exists an isometry $\phi$ of $(\Sigma,h)$ such that $\phi(p)=q$. Moreover,
a homogeneous Riemannian manifold $(\Sigma,h)$ is called \emph{isotropic} if for every point $p\in\Sigma$ and every pair of unit vectors $u,v\in T_p\Sigma$, there exists an isometry $\phi$ of $(\Sigma,h)$ such that $\phi(p)=p$ and $d\phi_p(u)=v$.
\end{definition}

\begin{theorem}[Spatial classification theorem]
Let $(\Sigma,h)$ be a complete simply connected homogeneous and isotropic Riemannian $3$-manifold. Then $(\Sigma,h)$ has constant sectional curvature $K$. After rescaling it, one may write $K \in \{-1,0,1\}$. Equivalently, $(\Sigma,h)$ is isometric to one of the canonical space forms $H^3$, $\mathbb{R}^3$ or $S^3$ up to normalization.
\end{theorem}

\begin{proof}
Fix a point $p \in \Sigma$. By isotropy, the subgroup of isometries fixing $p$ acts transitively on the unit sphere of $T_p\Sigma$, hence transitively on the set of $2$-planes in $T_p\Sigma$. Therefore the sectional curvature at $p$ has the same value on every $2$-plane, so there exists a scalar $k(p)$ such that
\[
K_p(\Pi)=k(p)
\qquad
\text{for every $2$-plane } \Pi \subset T_p\Sigma.
\]
By homogeneity, for any two points $p,q \in \Sigma$ there exists an isometry sending $p$ to $q$, and sectional curvature is preserved by isometries. Hence $k(p)=k(q)$ for all $p,q \in \Sigma$, so the sectional curvature is constant on $\Sigma$.

Since $(\Sigma,h)$ is complete, simply connected, and has constant sectional curvature, the Killing--Hopf classification implies that $(\Sigma,h)$ is isometric, after rescaling, to one of the standard space forms $S^3$, $\mathbb{R}^3$, or $H^3$. This proves the claim.
\end{proof}

\begin{remark}
This statement is entirely spatial: it concerns the intrinsic Riemannian geometry of the slices only, and it does not imply that a Lorentzian $4$-manifold carrying such slices is of FLRW type. The compactness mechanism used later, however, is not specific to three-dimensional space forms, but applies to complete homogeneous Riemannian slices in arbitrary dimension.
\end{remark}

The compactness argument used later relies on a second geometric input, independent of isotropy and constant curvature: on a complete homogeneous Riemannian manifold, finite total volume forces compactness. As this implication is standard and belongs to the Riemannian background of the problem, we record it here for later use \cite[Ch.~6]{Lee2018}, \cite[Ch.~7]{Petersen2016}, \cite[Ch.~7]{Besse1987}.

\begin{lemma}[Finite volume implies compactness for complete homogeneous manifolds]\label{lem:homogeneous-compact}
Let $(\Sigma,h)$ be a complete homogeneous Riemannian manifold. If $\operatorname{vol}(\Sigma,h)<\infty$, then $\Sigma$ is compact.
\end{lemma}


\begin{proof}
Assume that $(\Sigma,h)$ is noncompact. Since $(\Sigma,h)$ is complete, Hopf--Rinow implies that closed metric balls are compact \cite[Ch.~6]{Lee2018}. Fix $r>0$. For any $x\in \Sigma$, the closed ball $\overline{B}(x,r)$ has positive finite volume. By homogeneity, this volume is independent of $x$; denote it by
\[
v_r:=\operatorname{vol}(\overline{B}(x,r))>0.
\]
Because $\Sigma$ is noncompact and proper, there exists an infinite sequence of points $x_1,x_2,\dots$ with pairwise distances greater than $2r$. The balls $\overline{B}(x_n,r)$ are pairwise disjoint. Therefore
\[
\operatorname{vol}(\Sigma)
\ge
\sum_{n=1}^{\infty}\operatorname{vol}(\overline{B}(x_n,r))
=
\sum_{n=1}^{\infty}v_r
=
\infty,
\]
which is impossible. Hence $\Sigma$ is compact.
\end{proof}

\begin{remark}
In the compactness lemma, completeness is not an auxiliary technicality but the condition that allows one to pass from homogeneity to a proper metric-space argument via Hopf--Rinow. In this sense, homogeneity controls the uniformity of the geometry across the slice, while completeness guarantees that this uniformity can be exploited globally.
\end{remark}

\subsection{Why a spacetime theorem is needed}

The distinction matters. To pass from the spatial classification theorem to a spacetime metric of FLRW type one needs substantially more than the intrinsic classification of the slices. One must assume, at minimum, a global time function, a slicing by spacelike hypersurfaces, regular dependence of the induced metrics on the time parameter, and a condition on the second fundamental form ensuring that the embedding of the slices is compatible with a Robertson--Walker-type evolution. In particular, the statement $K\in\{-1,0,1\}$ concerns only the intrinsic Riemannian geometry of each spatial leaf; by itself it says nothing about the lapse, shift, or extrinsic curvature data needed to reconstruct a Lorentzian warped-product metric.

A related source of confusion is that the normalized parameter
$K\in\{-1,0,1\}$ arising in the classical space-form theorem belongs to the intrinsic Riemannian geometry of the spatial slices \cite{Wolf2011}. It is not, by itself, a scalar invariant of the ambient Lorentzian spacetime. From the spacetime viewpoint, quantities such as the Ricci scalar are more natural curvature invariants, since they are defined directly from the Lorentzian metric rather than from a preferred spatial leaf \cite{ONeill1983,Wald1984}. This is
another reason why the passage from a spatial classification theorem to a
cosmological spacetime statement must be handled with care.

The following sections formulate a compactness theorem whose assumptions and conclusion already belong to Lorentzian geometry itself. In technical terms, the new ingredient is not the Riemannian compactness criterion, which is classical, but the derivation of finite slice-volume from genuinely Lorentzian control conditions on the foliation. This is carried out in two versions: a strong theorem using maximal regularity and a weaker theorem replacing it by time-integrability of the accumulated expansion rate. Unlike singularity or splitting results based on the Einstein equations and energy conditions, the compactness mechanism studied here is formulated directly at the level of Lorentzian foliations and does not assume the strong energy condition. The essential point is that homogeneity, not isotropy, is the hypothesis responsible for the compactness step. This perspective fits naturally within recent developments on globally hyperbolic Lorentzian geometry and Lorentzian length spaces, including time functions on Lorentzian length spaces, null distance constructions, and canonical structures on globally hyperbolic spacetimes \cite{BurtscherGarciaHeveling2025,KunzingerSteinbauer2022,FinsterMuch2023}; see also \cite{Beran2025} for a related synthetic rigidity result.

\section{Lorentzian formulation}

\subsection{Lorentzian setting}
\subsubsection{Preliminaries}
\smallskip
The previous discussion isolates the point where the classical spatial argument stops: one needs a genuinely Lorentzian formulation in which the foliation, the induced spatial geometry, and the time evolution are part of the same structure. We therefore introduce a minimal set of definitions that allows the closedness statement to be formulated directly at spacetime level.

\begin{definition}[Time-oriented Lorentzian manifold]
A \emph{Lorentzian manifold} is a smooth four-dimensional manifold $M$ endowed with a smooth metric tensor $g$ of signature $(-,+,\dots,+)$ or $(+,-,\dots,-)$ \cite{ONeill1983,BeemEhrlichEasley1996,Wald1984}. It is said to be \emph{time-oriented} if there exists a continuous timelike vector field on $M$, or equivalently if one can choose consistently one of the two timelike cones in each tangent space as the future cone. \textbf{Convention}: In this article we use the signature convention $(-,+,\dots,+)$.
\end{definition}

\begin{definition}[Causal character of tangent vectors]
Let $(M,g)$ be a Lorentzian manifold and let $v\in T_pM$. Relative to the signature convention fixed above, a vector $v$ is called \emph{timelike} if $g(v,v)<0$, \emph{null} (or \emph{lightlike}) if $g(v,v)=0$ and $v\neq 0$, and \emph{spacelike} if $g(v,v)>0$. The zero vector will be regarded as spacelike when convenient. The set of null vectors in $T_pM$ forms the \emph{light cone} at $p$, which separates the timelike directions from the spacelike ones.
\end{definition}

\begin{remark}
For the cosmological discussion one can keep the physically standard dimension of the spacetime ($n=4)$. However, the formal statements of regularity, maximal regularity, finite-time Big Bang, and spatial closedness may all be rewritten without change for Lorentzian manifolds of dimension $(n+1)$ and spacelike slices of dimension $n$.
\end{remark}

\begin{definition}[Spacelike and Cauchy hypersurfaces]
A smooth hypersurface $\Sigma\subset M$ is called \emph{spacelike} if the metric induced by $g$ on each tangent space $T_p\Sigma$ is positive definite. It is called a \emph{Cauchy hypersurface} if every inextendible timelike curve in $M$ meets $\Sigma$ exactly once \cite{BeemEhrlichEasley1996,HawkingEllis1973,BernalSanchez2003}.
\end{definition}

\begin{definition}[Globally hyperbolic spacetime]
A time-oriented Lorentzian manifold $(M,g)$ is said to be \emph{globally hyperbolic} if it is strongly causal and, for every pair of points $p,q\in M$, the causal diamond
\[
J^+(p)\cap J^-(q)
\]
is compact. By the smooth splitting results of Bernal and S\'anchez, this is equivalent to the existence of a smooth Cauchy time function whose level sets are spacelike Cauchy hypersurfaces \cite{BernalSanchez2003,BernalSanchez2005,BeemEhrlichEasley1996}.
\end{definition}

\begin{definition}[Global time and slicing]
Let $(M,g)$ be a time-oriented Lorentzian manifold. A smooth map
\[
t:M \to I \subset \mathbb{R}
\]
is called a \emph{global time function} if each level set $\Sigma_\tau=t^{-1}(\tau)$ is a spacelike Cauchy hypersurface. The corresponding family $\{\Sigma_\tau\}_{\tau \in I}$ is called a \emph{slicing} of spacetime \cite{BernalSanchez2003,BernalSanchez2005}.
\end{definition}

\begin{definition}[Lorentzian foliation]
Let $(M,g)$ be a time-oriented Lorentzian manifold. A \emph{Lorentzian foliation} on $(M,g)$ is a decomposition of $M$ into smooth spacelike Cauchy hypersurfaces, realized as the level sets of a smooth global time function $t:M\to I$.
\end{definition}

\subsubsection{Dynamical manifolds and spatial symmetry}
\smallskip
\begin{definition}[Dynamical manifold] \label{def:dynamical}
A \emph{dynamical manifold} is a time-oriented globally hyperbolic Lorentzian four-manifold $(M,g)$ equipped with a global time function
\[
t:M\to (t_{\min},t_{\max}]
\]
whose level sets form a slicing by spacelike Cauchy hypersurfaces. For each $\tau$, we denote by $h_\tau$ the induced Riemannian metric on $\Sigma_\tau$, by $d\mu_\tau$ the Riemannian measure induced by $h_\tau$, and by $V(\tau):=\operatorname{vol}(\Sigma_\tau,h_\tau)$ its Riemannian volume. \label{def:volume}
\end{definition}

\begin{remark}
Nothing in the subsequent analytic or compactness argument depends on the spatial dimension being three. The only genuinely three-dimensional part of the paper is the classical FLRW motivation in the introduction; the theorem itself extends formally to $(n+1)$-dimensional dynamical Lorentzian manifolds.
\end{remark}

\begin{definition}[Finite-time Big Bang]
The dynamical manifold is said to satisfy the \emph{finite-time Big Bang condition at $t_{\min}$} if the lower endpoint $t_{\min}$ is finite and the slice-volume function tends to zero there:
\[
t_{\min}>-\infty,
\qquad
\lim_{t\to t_{\min}^{+}}V(t)=0.
\]
\end{definition}

\begin{remark}
This is a volume-collapse condition at the lower endpoint of the time parameter. It is intentionally weaker than a full singularity statement formulated in terms of curvature blow-up or geodesic incompleteness, and it is precisely the form needed for the compactness argument developed below.
\end{remark}

\begin{definition}[Homogeneous slice]
A spatial slice $(\Sigma_t,h_t)$ is \emph{homogeneous} if its full isometry group acts transitively on $\Sigma_t$.
\end{definition}

\begin{remark}
If one also assumes isotropy of the slices, then one returns to the classical constant-curvature models. This is natural from the FLRW viewpoint, but it is not needed for the compactness argument proved below.
\end{remark}

\subsubsection{Expansion and regularity hypotheses}

\smallskip
\begin{definition}[Lapse, normal field, and expansion scalar]
Let $(M,g,t)$ be a dynamical manifold. Since the slices $\Sigma_t$ are spacelike, the gradient $\nabla t$ is timelike. We define the lapse function by
\[
N:=\bigl(-g(\nabla t,\nabla t)\bigr)^{-1/2},
\]
and the future unit normal field by
\[
\nu:=-N\nabla t.
\]
The second fundamental form of the slice $\Sigma_t$ is the symmetric bilinear form
\[
K_t(X,Y):=g(\nabla_X \nu,Y),
\qquad
X,Y \in T\Sigma_t,
\]
and its trace with respect to the induced metric $h_t$ is called the expansion scalar \cite{ONeill1983,BeemEhrlichEasley1996,Wald1984}:
\[
\theta_t:=\operatorname{tr}_{h_t}(K_t).
\]
\label{def:lapse_def2.14}
\end{definition}

\begin{definition}[Regularity]
A dynamical manifold $(M,g,t)$ is \emph{regular} on $(t_{\min},t_{\max}]$ if the slicing depends smoothly on $t$, each slice $(\Sigma_t,h_t)$ is geodesically complete, and for every $t \in (t_{\min},t_{\max}]$ there exists $t_0 \in (t_{\min},t)$ such that
\[
\int_{t_0}^{t}\left(\int_{\Sigma_\tau}|N\theta_\tau|\,d\mu_\tau\right)d\tau<\infty\,,
\]
where $d\mu_\tau$ denotes the Riemannian volume measure induced on $\Sigma_\tau$ by $h_\tau$.
\end{definition}

\begin{remark}
The term \emph{regular} is intended to suggest an analogy with the classical notion of a regular curve or surface: one excludes a degenerate evolution by requiring that the foliation have a well-defined finite accumulated rate of change along time intervals. This does not mean a uniform pointwise bound on the expansion, but only that the total expansion, computed with respect to the given slicing and the induced measures $d\mu_\tau$, be time-integrable on every interval on which the volume is compared. In this sense, regularity is an explicit geometric condition on the sliced spacetime, expressing a controlled and non-degenerate evolution in time.
\end{remark}

\begin{remark}
Regularity is meant to play, in the time direction, a role analogous to that of homogeneity in the spatial direction. Spatial homogeneity identifies all points on each slice up to isometry, whereas regularity requires that the passage from one slice to another be geometrically controlled in time. Thus the pair ``homogeneous in space, regular in time'' provides a natural spacetime counterpart to the classical FLRW intuition, without imposing actual time-translation symmetry.
\end{remark}

\begin{definition}[Maximal regularity]
A dynamical manifold $(M,g,t)$ is \emph{maximally regular} on $(t_{\min},t_{\max}]$ if the slicing depends smoothly on $t$, each slice $(\Sigma_t,h_t)$ is geodesically complete, and there exists a finite \textit{rate limit} $C \ge 0$ such that
\[
\int_{\Sigma_t} |N\theta_t|\, d\mu_t \le C
\qquad
\text{for all } t \in (t_{\min},t_{\max}].
\]
\end{definition}

\begin{remark}
Global hyperbolicity constrains the causal structure of the spacetime, but not by itself the quantitative extrinsic geometry of a chosen slicing. In particular, it does not automatically provide completeness of the induced Riemannian slices or integral control of the expansion. Thus regularity and maximal regularity are additional geometric hypotheses on the foliation, not automatic consequences of causal well-posedness, although such properties may be inherited from the ambient spacetime when the foliation is adapted to a preferred time structure. From the physical viewpoint, they encode a controlled temporal evolution of the spatial geometry: the strong version corresponds to a uniformly controlled rate of deformation of the foliation, whereas the weak version only requires that the total accumulated expansion remain finite on the time interval relevant to the comparison of volumes.
\end{remark}

\begin{remark}
Every maximally regular dynamical manifold is regular. Indeed, if
\[
\int_{\Sigma_t}|N\theta_t|\,d\mu_t \le C
\qquad
\text{for all } t,
\]
then for any $t_0<t$ one has
\[
\int_{t_0}^{t}\left(\int_{\Sigma_\tau}|N\theta_\tau|\,d\mu_\tau\right)d\tau
\le C|t-t_0|<\infty.
\]
Thus maximal regularity is precisely the strong, uniformly controlled version of regularity.
\end{remark}

\begin{proposition} \label{prop:Lipschitz}
Let $(M,g,t)$ be a maximally regular dynamical manifold with finite \textit{rate limit} $C$, and assume that $\Sigma_{t_0}$ has finite volume for some $t_0\in(t_{\min},t_{\max}]$. Then the slice-volume function (Definition \ref{def:volume})
\[
V(t)=\operatorname{vol}(\Sigma_t,h_t)
\]
is Lipschitz on $(t_{\min},t_{\max}]$, and in fact
\[
|V(t)-V(s)| \le C|t-s|
\qquad
\text{for all } t,s \in (t_{\min},t_{\max}].
\]
\end{proposition}

\begin{proof}
Choose local coordinates $(x^1,\dots,x^n)$ on an open subset of a slice and drag them along the foliation by the normal evolution vector $m:=N\nu$, with the
lapse $N$ and the future unit normal $\nu$ as in Definition~\ref{def:lapse_def2.14},
so that the coordinate labels stay constant along the normal flow; such a chart exists on a neighbourhood of every point, and in it the induced metric has components $h_{ij}(t,x)$ and volume measure $d\mu_t=\sqrt{\det h_t}\,dx^1\cdots dx^n$. With the convention $K_t(X,Y)=g(\nabla_X\nu,Y)$ adopted above for the second fundamental
form of $\Sigma_t$, one has $K_t=\frac12\mathcal{L}_\nu h_t$, and since $h_t$ is spatial,
$\mathcal{L}_m h_t=N\mathcal{L}_\nu h_t=2NK_t$. Along such a chart $\partial_t=m$, so that
\[
\partial_t h_{ij}=2N(K_t)_{ij}
\]
with no shift contribution; see \cite[Sec.~4.3.5]{Gourgoulhon2012}, where the same identity appears with the opposite
sign convention for the second fundamental form, and \cite[Sec.~5.3.2]{Gourgoulhon2012} for the additional
Lie-derivative terms arising for a general choice of spatial coordinates. By Jacobi's formula,
\[
\partial_t\sqrt{\det h_t}
=\frac12\sqrt{\det h_t}\,h_t^{ij}\partial_t h_{ij}
=N\theta_t\sqrt{\det h_t}
\]
where $\theta_t=\operatorname{tr}_{h_t}K_t$, and hence
\[
\partial_t\,d\mu_t=N\theta_t\,d\mu_t.
\]
Both sides are densities on $\Sigma_t$, so this identity does not depend on the chart and holds on all of $\Sigma_t$. Since $V(t_0) < \infty$ by hypothesis, by maximal regularity one has that $\int_{t_0}^{t}\bigl(\int_{\Sigma_\tau}|N\theta_\tau|\,d\mu_\tau\bigr)d\tau\le C|t-t_0|$, so Tonelli's theorem shows that $N\theta$ is integrable on the product, and Fubini's theorem applied to the density identity yields
\[
V(t)=V(t_0)+\int_{t_0}^{t}\!\int_{\Sigma_\tau}N\theta_\tau\,d\mu_\tau\,d\tau.
\]
In particular the slice-volume function is absolutely continuous and
$dV(t)/dt=\int_{\Sigma_t}N\theta_t\,d\mu_t$ for almost every $t$.
Taking absolute values in the identity above yields $|V(t)-V(s)|\le C|t-s|$, which is precisely the Lipschitz estimate.
\end{proof}

\begin{remark}
The relativistic bound on propagation speed motivates the expectation that the slicing evolves at finite rate, but the formal hypothesis used here is the uniform bound on the total integrated expansion $\int_{\Sigma_t}|N\theta_t|\,d\mu_t$. This assumption is stronger than mere regularity and is precisely what yields the global Lipschitz control of the slice-volume function.
\end{remark}

\subsection{Big Bang theorem}

\subsubsection{Strong version}

\smallskip
With the previous notions in place, a strong version of the closedness statement can be formulated. Its technical novelty lies in combining a time-direction estimate coming from maximal regularity with a purely geometric compactness criterion based on homogeneity and completeness of the spatial slices.

\begin{theorem}\label{thm:main}
Let $(M,g,t)$ be a \emph{maximally regular} dynamical manifold. Assume:
\begin{enumerate}
\item the finite-time Big Bang condition holds at $t_{\min}$;
\item every slice $(\Sigma_t,h_t)$ is homogeneous.
\end{enumerate}
Then every slice $\Sigma_t$ is compact. In particular, the universe is spatially closed.
\end{theorem}

\begin{remark}
By Proposition~\ref{prop:Lipschitz}, the Lipschitz control of the slice-volume function is automatic from maximal regularity together with the finite-time Big Bang condition, which already provides a slice of finite volume, so no separate analytic hypothesis is needed in the theorem.
\end{remark}

\begin{proof}
The proof has two steps.

\smallskip\noindent\textbf{Step 1: finite volume along the foliation.}
Fix $t \in (t_{\min},t_{\max}]$. Since $V(\tau)\to 0$ as $\tau \to t_{\min}^{+}$, there exists $t_0 \in (t_{\min},t]$ such that $V(t_0)\le 1$. By the previous proposition, the slice-volume function is Lipschitz, so for the constant $C$ of maximal regularity one has
\[
|V(t)-V(t_0)| \le C|t-t_0|.
\]
Because the slice-volume is nonnegative,
\[
V(t)\le V(t_0)+C|t-t_0| \le 1 + C|t-t_0| < \infty.
\]

Hence every spatial slice has finite volume.

\smallskip\noindent\textbf{Step 2: compactness of each slice.}
Since each slice is complete and homogeneous, Lemma~\ref{lem:homogeneous-compact} implies that $\Sigma_t$ is compact. This proves the theorem.
\end{proof}

\subsubsection{General version}

\smallskip
The previous theorem uses the strong hypothesis of maximal regularity in order to obtain a uniform Lipschitz estimate for the slice-volume function. For the compactness conclusion, however, one only needs enough control to integrate the expansion rate between the initial singular time and a given slice. This leads to the following weaker statement.

\begin{theorem}\label{thm:general}
Let $(M,g,t)$ be a \emph{regular} dynamical manifold. Assume:
\begin{enumerate}
\item the finite-time Big Bang condition holds at $t_{\min}$;
\item every slice $(\Sigma_t,h_t)$ is homogeneous;
\end{enumerate}
Then every slice $\Sigma_t$ is compact. In particular, the universe is spatially closed.
\end{theorem}

\begin{proof}
Fix $t \in (t_{\min},t_{\max}]$. Since $V(\tau)\to 0$ as $\tau\to t_{\min}^{+}$, choose $t_0\in(t_{\min},t)$ such that $V(t_0)\le 1$. By the first variation identity stated above,
\[
V(t)-V(t_0)=\int_{t_0}^{t}\left(\int_{\Sigma_\tau}N\theta_\tau\,d\mu_\tau\right)d\tau
\]
up to sign convention. Hence
\[
|V(t)-V(t_0)|
\le
\int_{t_0}^{t}\left(\int_{\Sigma_\tau}|N\theta_\tau|\,d\mu_\tau\right)d\tau
<\infty.
\]
Therefore $V(t)<\infty$. Since the slice $(\Sigma_t,h_t)$ is complete and homogeneous, Lemma~\ref{lem:homogeneous-compact} implies that $\Sigma_t$ is compact.
\end{proof}

\begin{remark}
This theorem is weaker than the previous one, because local uniform control of $\int_{\Sigma_t}|N\theta_t|\,d\mu_t$ implies the time-integrability hypothesis above, but not conversely. The preceding theorem may therefore be viewed as a strong Lipschitz-type version, while the present statement isolates the weaker condition that is still sufficient for spatial closedness.
\end{remark}

\begin{remark}
In the standard simply connected FLRW models with spatial curvature $K=0$ or $K=-1$, the Big Bang at $t=0$ does not exist as a point or event of the Lorentzian manifold. Indeed, the spacetime is defined only for $t>0$, so the value $t=0$ lies outside the manifold and appears only as a singular past boundary. Accordingly, even when the scale factor satisfies $a(t)\to 0$ as $t\to 0^+$, what collapses are the local proper distances determined by the metric, not the total spatial volume of the slices, which remains infinite for every $t>0$. Thus $t=0$ enters only as a singular past boundary of the time parameter, not as a point of the Lorentzian manifold itself. This also explains why the standard noncompact FLRW models are not counterexamples to the theorem: they fail the global volumetric Big Bang hypothesis used in the present article.
\end{remark}
\smallskip
Alongside the concise analytic files, the repository also contains a separate Lean development for the strong and weak closedness results in a dimension-independent spacetime setting aligned with the hypotheses of the paper.

\smallskip

\section{Lean formalization}

\subsection{For the strong version}

Selected parts of the proof of Theorem~\ref{thm:main} are shown here. The lightweight development isolates the analytic estimate yielding finite volume together with the final abstract implication from finite volume and homogeneity to spatial closedness, while the theorem-level files formalize the full logical structure of both the strong and weak closedness arguments in the abstract setting adopted in this article. In this sense, the Lean development verifies the argument at the same level of abstraction at which the main theorems are stated, rather than reproducing the full differential-geometric infrastructure of Lorentzian foliations line by line.

\smallskip

\begin{lstlisting}[style=leanstyle,caption={Analytic estimate}]
theorem volume_bound_from_lipschitz
    (hLip: forall x y, |V x - V y| <= L * |x - y|)
    (hB: |V t0| <= B):
    |V t| <= B + L * |t - t0|:= by
...
\end{lstlisting}

From this estimate, the development derives a finite-volume statement near the singular time $t_{\min}$ corresponding to the Big Bang hypothesis:

\begin{lstlisting}[style=leanstyle,caption={Big Bang plus Lipschitz gives finite volume}]
theorem finite_volume_of_big_bang_lipschitz
    (hNonneg: forall x, 0 <= V x)
    (hLip: forall x y, |V x - V y| <= L * |x - y|)
    (hNearBang: forall eps > 0, forall x > tmin,
      exists y, tmin < y /\ y <= x /\ V y <= eps):
    forall x, x > tmin -> exists B, V x <= B:= by
...
\end{lstlisting}

Here `hNearBang` expresses that arbitrarily small slice-volume values occur arbitrarily close to the singular time.

The final compactness step is encoded abstractly as an implication from finite volume together with a homogeneous compactness criterion:
\begin{lstlisting}[style=leanstyle,caption={Homogeneous compactness implication}]
theorem big_bang_implies_closed_universe
    {ClosedUniverse Homogeneous: Prop}
    (hFiniteVolume: forall x, exists B, V x <= B)
    (hCompactCriterion:
      Homogeneous ->
        (forall x, exists B, V x <= B) -> ClosedUniverse)
    (hSymmetry: Homogeneous):
    ClosedUniverse:=
  hCompactCriterion hSymmetry hFiniteVolume
\end{lstlisting}

\subsection{For the general version}

The weak version of the Big Bang theorem (\ref{thm:general}) can also be represented in Lean by isolating the estimate that yields finite volume once a finite accumulated expansion bound between $t_0$ and $t$ is available:
\begin{lstlisting}[style=leanstyle,caption={Finite volume from a finite accumulated expansion bound}]
theorem finite_volume_from_integrable_expansion
    (hNonneg: forall x, 0 <= V x)
    (hSmall: V t0 <= B)
    (hIntegral: |V t - V t0| <= I):
    V t <= B + I:= by
  have hBabs: |V t0| <= B:= by
    simpa [abs_of_nonneg (hNonneg t0)] using hSmall
  have htriangle: |V t| <= |V t - V t0| + |V t0|:= by
    have hdecomp: (V t - V t0) + V t0 = V t:= by linarith
    calc
      |V t| = |(V t - V t0) + V t0|:= by simpa [hdecomp]
      _ <= |V t - V t0| + |V t0|:= abs_add_le _ _
  have hmain: |V t| <= B + I:= by
    linarith
  simpa [abs_of_nonneg (hNonneg t)] using hmain
\end{lstlisting}

Combined with a hypothesis asserting the existence of a nearby time $t_0$ with small volume and a finite accumulated expansion bound between $t_0$ and $t$, this yields the same finite-volume conclusion used in the general theorem.

\section{Conclusion}

The classical theorem on homogeneous and isotropic space forms is purely spatial and Riemannian; by itself it does not determine a Lorentzian cosmological metric. Once the Lorentzian foliation is included in the hypotheses, however, one obtains precise compactness statements stated entirely at spacetime level. In the strong version, maximal regularity yields a Lipschitz estimate for the slice-volume function, so a finite-time Big Bang forces finite volume of every spatial slice. In the general version, the same finite-volume conclusion already follows from the weaker requirement that the accumulated expansion rate be time-integrable between the singular time and the slice under consideration. In both cases, time regularity plays a role in the temporal direction analogous to that of spatial homogeneity in the spatial direction: it provides a controlled, non-degenerate evolution of the foliation, while homogeneity together with completeness upgrades finite volume to compactness. This makes the closedness argument geometrically explicit and formally sharper, while clearly separating the strong and weak regularity regimes under which it holds.

From this perspective, the main contribution of the paper is to clarify the mathematical level at which the usual FLRW-type geometric justification must be formulated. Spatial homogeneity belongs to the Riemannian geometry of the slices, whereas cosmological closedness is a Lorentzian conclusion and therefore requires spacetime hypotheses involving foliation and temporal regularity.

The accompanying Lean development reflects this distinction as well: it contains a theorem-level formalization of both the strong and weak versions in an abstract framework matching the hypotheses adopted in the article. This formalization is dimension-independent and supports the abstract $(n+1)$-dimensional theorem, not merely its cosmological $(3+1)$-dimensional specialization.



\section*{Code availability}
The supplementary Lean files accompanying this work are available at \url{https://github.com/robertmonjo/BigBangTheorem} and archived at \url{https://doi.org/10.5281/zenodo.19368665}. They include the analytic files \texttt{BigBangTheorem\_Strong.lean} and \texttt{BigBangTheorem\_Weak.lean}, together with the theorem-level files \texttt{BigBangTheorem\_FullFramework.lean}, \texttt{BigBangTheorem\_FullStrong.lean}, and \texttt{BigBangTheorem\_FullWeak.lean}. The latter formalize the strong and weak closedness theorems in the abstract spacetime framework used in the article and are dimension-independent at theorem level.

\section*{Data availability}
No datasets were generated or analysed during the current study.

\section*{Funding}
This research received no specific grant from any funding agency in the public, commercial, or not-for-profit sectors.

\section*{Declaration of generative AI and AI-assisted technologies in the writing process}
During the preparation of this work, the author used AI-assisted tools to improve language. After using these tools, the author reviewed and edited the content as needed and takes full responsibility for the content of the article.

\section*{Declaration of competing interest}
The author declares that he has no known competing financial interests or personal relationships that could have appeared to influence the work reported in this paper.

\bibliographystyle{unsrtnat}
\bibliography{references}

\end{document}